\documentclass[a4paper,12pt]{article}
\usepackage{amsmath,amsthm,amssymb,amsfonts,mathbbol}
\usepackage{bbm,bm}
\usepackage{latexsym}
\usepackage{mathrsfs}
\usepackage{threeparttable}
\usepackage{tabularx}
\usepackage{booktabs}
\usepackage{graphicx,subfigure}
\usepackage{color}
\usepackage{amsthm,amsmath,amssymb}
\usepackage{mathrsfs}
\usepackage{cite}
\usepackage{indentfirst}
\usepackage{geometry}
\usepackage{caption}
\usepackage{lineno}
\usepackage{abstract}
\usepackage{enumerate,mdwlist}
\usepackage[numbers,sort&compress]{natbib}
\usepackage{epstopdf}

\def \[{\begin{equation}}
\def \]{\end{equation}}

\theoremstyle{plain}
\newtheorem{The}{Theorem}[section]
\newtheorem{Lem}[The]{Lemma}

\newtheorem{Pro}[The]{Proposition}

\begin{document}

\setlength{\baselineskip}{20pt}
\renewcommand{\thefootnote}{\fnsymbol{footnote}}
\setcounter{footnote}{0}
\begin{center}{\Large \bf On the $d$-transversal number of cylindrical and toroidal grids}\footnote{This work is supported by NSFC\,(Grant No. 12271229).}

\vspace{4mm}

{Hailun Wu, Heping Zhang \footnote{The corresponding authors. E-mail address:
zhanghp@lzu.edu.cn (H. Zhang).}
}

\vspace{2mm}

\footnotesize{School of Mathematics and Statistics, Lanzhou University, Lanzhou, Gansu 730000, P. R. China}

\end{center}
\noindent {\bf Abstract}:
For a positive integer $d$, 
a $d$-transversal set of a graph $G$ is an edge subset $T\subseteq E(G)$ 
such that $|T\cap M|\geq d$ for every maximum matching $M$ of $G$. 
The $d$-transversal number of $G$, denoted by $\tau_d(G)$, 
is the minimum cardinality of a $d$-transversal set in $G$. 
It is NP-complete to determine the $d$-transversal number of a bipartite graph for any fixed $d\geq 1$. 
Ries et al. (Discrete Math. 310 (2010) 132-146) established the $d$-transversal number 
of rectangular grids $P_m\square P_n$.
In this paper, we consider cylindrical grids $P_m\square C_n$ and toroidal grids $C_m\square C_n$. 
We derive explicit expressions for the $d$-transversal numbers of 
$P_m\square C_n$ for $m\geq 1$ and even $n\geq 4$, 
or even $m\geq 2$ and $n=3$, and of $C_m\square C_n$ with even order, for $1\leq d\leq \frac{mn}{2}$. 
For the other cases, 
we obtain explicit expressions or bounds for their $d$-transversal numbers.

\vspace{2mm} \noindent{\bf Keywords}: Transversal; Maximum matching; Stable set; 
Cylindrical grid; Toroidal grid
\vspace{2mm}

{\setcounter{section}{0}
\section{Introduction}\setcounter{equation}{0}

In this paper, we consider simple, finite and undirected graphs.
Let $G$ be a graph with vertex set $V(G)$ and edge set $E(G)$.
A graph is called \textit{even} or \textit{odd} based on whether it
contains an even or odd number of vertices.
A \textit{matching} in $G$ is a subset $M\subseteq E(G)$ such that no two edges in $M$ share a common endpoint.
A matching is \textit{maximum} if its cardinality is the largest among all matchings in $G$. 
The \textit{matching number} of $G$, denoted by $\nu(G)$, 
is the size of a maximum matching in $G$. 
For a positive integer $d$, a \textit{$d$-transversal set} of $G$ is a subset $T\subseteq E(G)$
satisfying $|T\cap M|\geq d$ for every maximum matching $M$ of $G$. 
The smallest cardinality of a $d$-transversal set is called the \textit{$d$-transversal number} of $G$, 
denoted by $\tau_d(G)$, where $1\leq d\leq \nu(G)$.

The concept of a $d$-transversal set for maximum matchings in graphs was introduced by
Zenklusen et al. \cite{ZRB2009}. In particular, 
when a graph $G$ has a perfect matching or an almost perfect matching, 
the 1-transversal number $\tau_1(G)$ reduces to the \textit{matching preclusion number} $mp(G)$ of $G$. 
The \textit{conditional matching preclusion number} $mp_1(G)$ of $G$ is the smallest number of edges 
whose deletion results in a graph with a decreased matching number and no isolated vertices. 
The (condtional) matching preclusion number serves as a measure
of robustness in the event of edge failure in interconnection networks and has been extensively studied. 
For example, see \cite{BH2005, CL2009, LLZ2023, CLL2012, CL2007, CL2012, LHZ2016}. 
Thus, the $d$-transversal number of a graph can be regarded as a natural generalization 
of matching preclusion in a different direction.

Additionally, the related $d$-transversal problem has been extensively studied 
for various graph substructures in recent years, including stable sets \cite{BCP2012, LFR2024}, 
cliques \cite{LCC2004}, decycling sets \cite{BSL2002}, matroid bases \cite{CM2011}, etc.
Zenklusen et al. \cite{ZRB2009} proved that finding a minimum $d$-transversal set 
in a bipartite graph is \textit{NP-complete} for any fixed $d\geq 1$. 
Furthermore, they derived explicit formulas for the $d$-transversal number 
in regular bipartite graphs and complete graphs.

\begin{The}\rm{\cite{ZRB2009}} 
Let $G$ be a $k$-regular bipartite graph of order $2n$.
Then $\tau_d(G)=kd$ for $1\leq d\leq n$.
\end{The}

\begin{The}\rm{\cite{ZRB2009}} 
Let $d, n\geq 1$ be integers with $2d\leq n$ and $r=\lfloor\frac{n}{2}\rfloor-d$. 
For the complete graph $K_n$, $\tau_d(K_n)=\tbinom{n}{2}-\tbinom{2r+1}{2}$
if $d\leq \lfloor\frac{n}{2}\rfloor-\frac{2n-3}{5}$, and $\tau_d(K_n)=\tbinom{n-r}{2}$ otherwise.
\end{The}

Subsequently, Ries et al. \cite{RBZ2010} presented an $O(n^5)$ polynomial-time 
algorithm for finding minimum $d$-transversal sets in a tree of order $n$, 
and established explicit formulas for the $d$-transversal number of 
rectangular grids $P_m\square P_n$ (see \cite{RBZ2010} for details). 
This motivates us to study the $d$-transversal numbers of cylindrical grids 
$P_m\square C_n$ and toroidal grids $C_m\square C_n$.

In statistical thermodynamics, 
to study the adsorption of diatomic molecules (dimers) on crystalline surfaces, 
physicists often seek to determine the number of dimer coverings in a rectangular region. 
Such a dimer covering corresponds to a perfect matching in a rectangular grid. 
Fisher \cite{F1961}, Temperley \cite{TF1961}, and Kasteleyn \cite{K1961} 
independently solved the dimer covering problem for rectangular grids $P_m\square P_n$. 
Later, McCoy and Wu \cite{MW1973} and Kasteleyn \cite{K1961} extended the enumeration 
of perfect matchings to cylindrical grids $P_m\square C_n$ and 
toroidal grids $C_m\square C_n$, respectively.

The rest of the paper is organized as follows. 
In Section 2, we introduce useful notations and terminology. 
We also present some properties of $d$-transversal sets in graphs and 
a condition for computing the $d$-transversal number of (non-bipartite) regular graphs (Proposition 2.3).
In Section 3, we consider cylindrical grids $P_m\square C_n$. 
The first two Subsections determine the $d$-transversal number of bipartite 
cylindrical grids $P_m\square C_n$ for $1\leq d\leq \frac{mn}{2}$. 
In Subsection 3.3, for even $m\geq 2$ and odd $n\geq 3$,
we determine $\tau_d(P_m\square C_n)$ for $1\leq d\leq n-1$ or $d=\frac{mn}{2}$. 
When $n\leq d< \frac{mn}{2}$, we give a sharp lower bound for $\tau_{d}(P_m\square C_n)$. 
Specifically, we derive an explicit formula to $\tau_d(P_m\square C_3)$ for $1\leq d\leq \frac{3m}{2}$. 
In Subsection 3.4, for odd $m\geq 3$ and odd $n\geq 3$, 
we determine $\tau_d(P_m\square C_n)$ when $\frac{m-3}{2}+n\leq d\leq \frac{mn-1}{2}$. 
For $\frac{m-1}{2}\leq d< \frac{m-5}{2}+n$, we establish a sharp upper bound for $\tau_{d}(P_m\square C_n)$. 
In Section 4, applying Proposition 2.3, we determine the $d$-transversal number of 
$C_m\square C_n$ with even order for $1\leq d\leq \frac{mn}{2}$. 
Furthermore, for odd $m\geq 3$ and odd $n\geq 3$, 
we obtain $\tau_d(C_m\square C_n)$ for $d=1$ or $\frac{m-1}{2}\leq d\leq \frac{mn-1}{2}$, 
and provide upper and lower bounds for $\tau_{d}(C_m\square C_n)$ in the remaining cases.

\section{Preliminaries}

For undefined graph-theoretic terminology and notation, we refer to \cite{BM2008, LM2024}.
Let $G$ be a graph. Given an edge subset $F\subseteq E(G)$, 
we denote by $G-F$ the subgraph obtained by deleting all edges in $F$ from $G$. 
For any nonempty proper subset $X$ of $V(G)$, 
let $\partial(X)$ denote the set of edges with one endpoint in $X$ and the other in $V(G)\backslash X$. 
In particular, for a vertex $v\in V(G)$, we write $\partial(v)$ instead of $\partial(\{v\})$ for simplicity.
For a vertex subset $X\subseteq V(G)$, $G[X]$ denotes the subgraph induced by $X$.
Given two subsets $X, Y \subseteq V(G)$, 
$E[X, Y]$ is the set of edges with one endpoint in $X$ and the other in $Y$. 
A \textit{spanning} subgraph of $G$ is a subgraph containing all vertices of $G$.
A graph is \textit{bipartite} if its vertex set can be partitioned into two classes 
$W$ and $B$ such that every edge connects a vertex in $W$ to a vertex in $B$. 
A \textit{path} and a \textit{cycle} with $k$ vertices are denoted by $P_k$ and $C_k$, respectively. 
The \textit{ceiling} function $\lceil x\rceil$ denotes the smallest integer not less than $x$, 
and the \textit{floor} function $\lfloor x\rfloor$ denotes the greatest integer not greater than $x$.

The \textit{Cartesian product} of two graphs $G$ and $H$ is the graph $G\square H$ 
whose vertex set is $V(G)\times V(H)$, and whose edge set consists of all pairs $(g_1, h_1)(g_2, h_2)$ 
such that either $g_1=g_2$ and $h_1h_2\in E(H)$, or $g_1g_2\in E(G)$ and $h_1=h_2$.
For integers $m\geq 1$ and $n\geq 3$, 
a \textit{cylindrical grid} is the Cartesian product $P_m\square C_n$. 
It is naturally embedded on a cylinder such that each face, 
except the two open ends, is bounded by a quadrilateral. 
More precisely, $P_m\square C_n$ is a graph with vertex set $\{x_{ij} |1\leq i\leq m$ and $1\leq j\leq n\}$,
its edge set consists of horizontal edges $h_{ij}=x_{ij}x_{i,j+1}$ 
for $1\leq i\leq m$ and $1\leq j\leq n$ (indices modulo $n$), 
and vertical edges $v_{ij}=x_{ij}x_{i+1,j}$ for $1\leq j\leq n$ and $1\leq i\leq m-1$. 
For example, $P_3\square C_4$ is illustrated in Fig.\hspace{0.06cm}\ref{A}(a).

For integers $m\geq 3$ and $n\geq 3$, a \textit{toroidal grid} is the Cartesian product $C_m\square C_n$, 
embedded on a torus with all faces bounded by quadrilaterals. 
Its vertex set $V(C_m\square C_n)$ coincides with $V(P_m\square C_n)$, 
and its edge set is $E(P_m\square C_n)\cup\{v_{mj}=x_{mj}x_{1j}|j=1,2, \dots, n\}$. 
An example $C_3\square C_4$ is shown in Fig.\hspace{0.06cm}\ref{A}(b).

\begin{figure}[!htbp]
  \centering
  \subfigure{
  \begin{minipage}{6cm}
  \centering
  \includegraphics[totalheight=3.7cm]{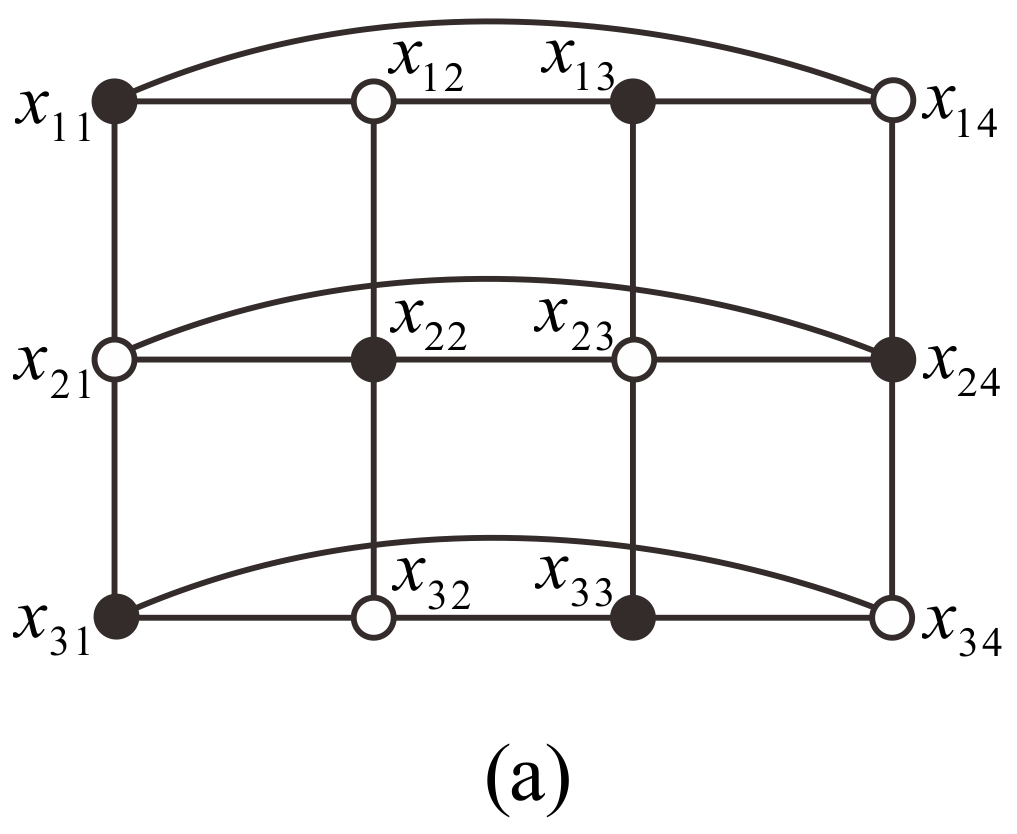}
   \end{minipage}%
   }
  \hspace{0.5cm}
  \subfigure{
  \begin{minipage}{6cm}
  \centering
  \includegraphics[totalheight=3.7cm]{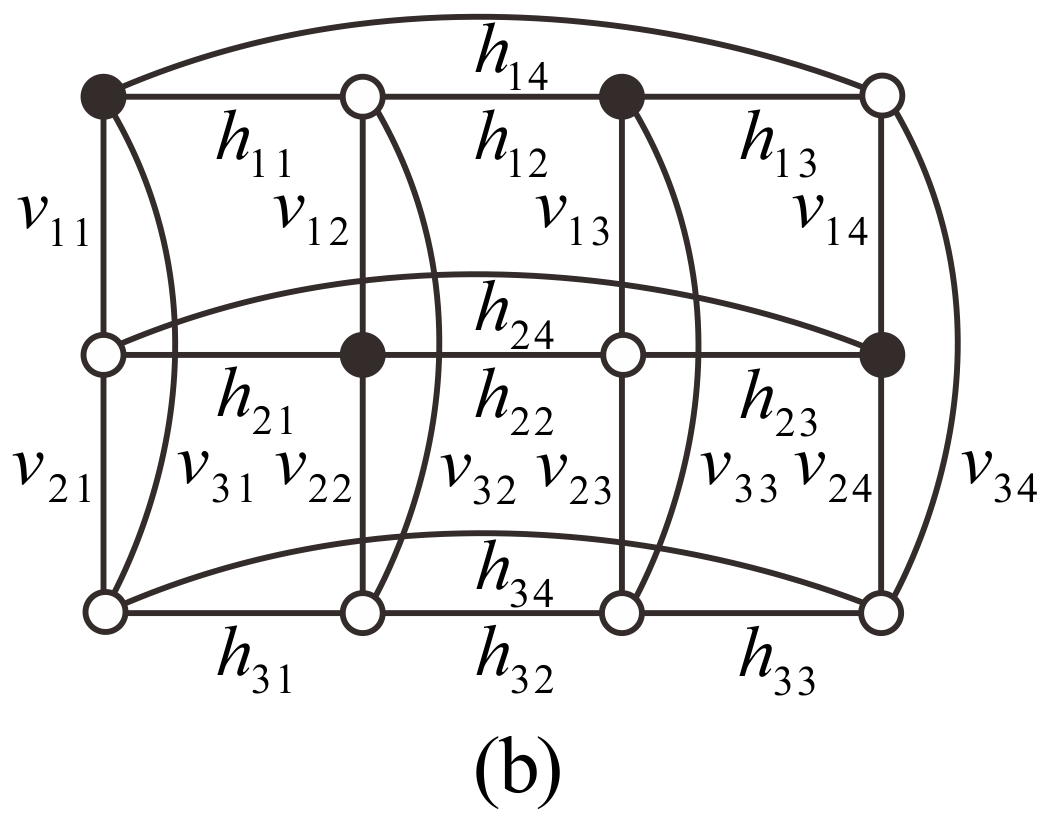}
   \end{minipage}
   }
  \caption{\label{A}\small{(a)$P_3\square C_4$, (b)$C_3\square C_4$.}}
\end{figure}

For every maximum matching $M$ of a graph $G$, let
$$\begin{aligned}
 F_0(G)=\{uv\in E(G)|uv\notin M\},~~ F_1(G)=\{uv\in E(G)|uv\in M\}.
\end{aligned}$$ 
Clearly, if $|F_1(G)|\geq d$, then $\tau_{d}(G)=d$.
Let $M$ be a matching of $G$. A vertex $v$ is \textit{covered} by $M$ if there exists an edge $uv\in M$. 
A matching in $G$ is \textit{perfect} if it covers all vertices of $G$. 
An \textit{almost perfect} matching is a matching covering all but one vertex of $G$. 
A vertex $v$ is \textit{strongly covered} if every maximum matching covers $v$. 
The set of all strongly covered vertices of $G$ is denoted by $S(G)$. 
Thus, if $G$ has a perfect matching, then $S(G)=V(G)$.


A \textit{stable set} $S$ of $G$ is a set of vertices no two of which are adjacent.
The maximum cardinality of a stable set in $G$ is called the \textit{stability number} of $G$, 
denoted by $\alpha(G)$. As we will see, 
stable sets play an important role in constructing a $d$-transversal set of $G$.


\begin{Pro}\rm{\cite{ZRB2009}}
Let $\{v_1, v_2, \dots, v_d\}$ be a stable set of $G$ such that each $v_i$ is strongly covered. 
Then $T=\partial(v_1)\cup \partial(v_2)\cup \dots \cup \partial(v_d)$ is a $d$-transversal set of $G$.
\end{Pro}

\begin{Pro}
Let $G$ be a graph with at least one edge. Then
$\tau_1(G)< \tau_2(G)< \cdots < \tau_{\nu(G)}(G)$.
\end{Pro}
\begin{proof}
For any $1\leq d\leq \nu(G)-1$, a $(d+1)$-transversal set of $G$ is also a $d$-transversal set, 
implying $\tau_{d}(G)\leq \tau_{d+1}(G)$. 
Let $T$ be a minimum $(d+1)$-transversal set of $G$. By definition, 
for any $uv\in T$, $T\backslash \{uv\}$ is a $d$-transversal set. 
Thus, $\tau_{d}(G)\leq |T\backslash \{uv\}|< \tau_{d+1}(G)$.
\end{proof}

\noindent\textbf{Remark.} The construction in Proposition 2.1 is not necessarily minimum. 
For example, consider the graph $G$ in Fig.\hspace{0.05cm}\ref{B}, 
which has a perfect matching and each edge is contained in some perfect matching.
Consequently, $S(G)=V(G)$ and $F_0(G)=F_1(G)=\emptyset$. Clearly, $\tau_1(G)=2$. 
By Proposition 2.2, $\tau_2(G)\geq 3$. Applying Proposition 2.1, 
$\partial(v_2)\cup \partial(v_3)$ forms a $2$-transversal set of size 4. 
Let $T=\{v_2v_5, v_3v_6, v_4v_7\}$. For any perfect matching $M$ of $G$, 
regardless of how $v_1$ is covered by $M$, we have $|M\cap T|=2$. 
Thus, $T$ is a minimum $2$-transversal set of size 3.

\begin{figure}[!htbp]
\begin{center}
\includegraphics[totalheight=3.5cm]{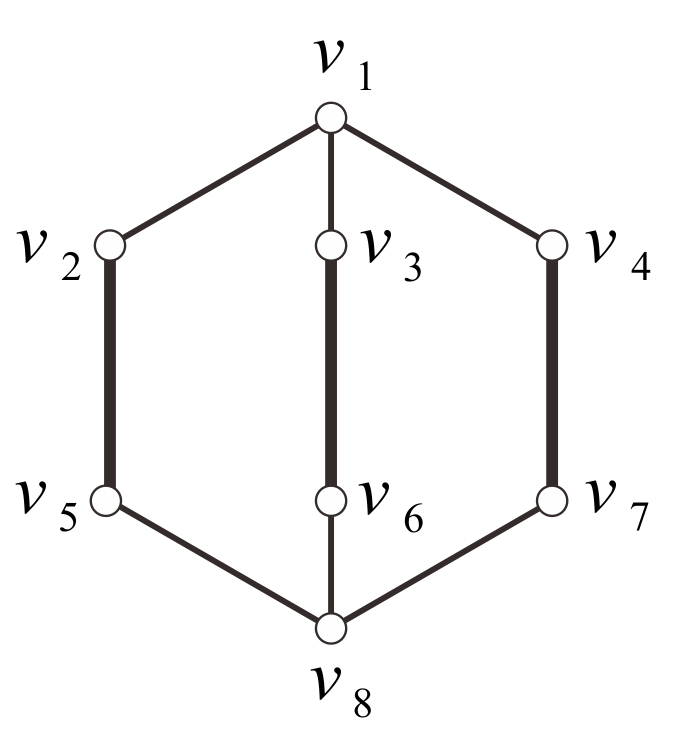}
\caption{\label{B}\small{A minimum $2$-transversal set $T$ (thick edges) of $G$.}}
\end{center}
\end{figure}

The next result provides the $d$-transversal number for regular graphs. 
It can also be considered as a generalization of Theorem 1.1.

\begin{Pro}
Let $G$ be a $k$-regular graph of order $2n$. 
If $G$ has $k$ disjoint perfect matchings and $\alpha(G)\geq \lceil \frac{n}{2}\rceil$, 
then $\tau_{d}(G)=kd$ for $1\leq d\leq n$.
\end{Pro}
\begin{proof}
Let $S$ be a maximum stable set of $G$. 
For $1\leq d\leq \lceil \frac{n}{2}\rceil$, 
construct a set $T$ by selecting the edge subset $\partial(S_1)$, 
where $S_1\subseteq S$ and $|S_1|=d$. 
By Proposition 2.1, $T$ is a $d$-transversal set. 
Thus, $\tau_{d}(G)\leq |T|=kd$. 
For $\lceil \frac{n}{2}\rceil\leq d\leq n$, 
let $G'$ be a spanning subgraph of $G$ satisfying $E(G)=\partial(S_2)$, 
where $S_2\subseteq S$ and $|S_2|=n-d$. 
Since $G$ contains perfect matchings, $\nu(G')= n-d$. 
Define $T'=E(G)\backslash \partial(S_2)$. 
For every perfect matching $M$ of $G$, we have $|T'\cap M|\geq d$. 
Thus, $T'$ is a $d$-transversal set with $|T'|=nk-k(n-d)=kd$. 
Hence, $\tau_{d}(G)\leq kd$. Finally, 
since $G$ contains $k$ disjoint perfect matchings, 
any $d$-transversal set must intersect each perfect matching in at least $d$ edges. 
This implies $\tau_{d}(G)\geq kd$.
\end{proof}

Note that the condition $\alpha(G)\geq \lceil \frac{n}{2}\rceil$ in Proposition 2.3 is necessary. 
Consider the complete graph $K_8$, we have $\alpha(K_8)=1<\lceil \frac{8}{4}\rceil=2$. 
Furthermore, $K_8$ admits $7$ disjoint perfect matchings \cite{ABS1990}. 
However, Theorem 1.2 implies $\tau_{2}(K_8)=15>14$.

\section{$d$-transversal number of $P_m\square C_n$}

We consider the following four cases based on the parity of $m$ and $n$.

\subsection{$m$ and $n$ are even}


\begin{Lem}
For even $m\geq2$ and even $n\geq 4$, 
$\tau_{d}(P_m\square C_n)=3d$ if $1\leq d\leq n$.
\end{Lem}
\begin{proof}
Let $1\leq d\leq n$. Consider a stable set 
$S=\{x_{11}, x_{13},\dots, x_{1,{n-1}}, x_{m2}, x_{m4},\dots, x_{mn}\}$ in $P_m\square C_n$. 
By Proposition 2.1, we construct a $d$-transversal set $T$ by selecting the edge subset $\partial(S')$, 
where $S'\subseteq S$ and $|S'|=d$. 
This implies $\tau_{d}(P_m\square C_n)\leq |T|=3d$. 
Furthermore, $P_m\square C_n$ contains three disjoint perfect matchings $M_1$, $M_2$, $M_3$ defined as: 
$M_1=\{h_{i1}, h_{i3}, \dots, h_{i,n-1}| i=1,2,\dots, m\}$, 
$M_2=\{h_{i2}, h_{i4}, \dots, h_{in}| i=1,2,\dots, m\}$, 
$M_3=\{v_{1j}, v_{3j}, \dots, v_{m-1,j}| j=1,2,\dots, n\}$. 
Since any $d$-transversal set $T'$ must satisfy $|T'\cap M_i|\geq d$ for $1\leq i\leq 3$, we have
$$\begin{aligned}
 \tau_{d}(P_m\square C_n)\geq \sum\limits_{i=1}\limits^{3}|T'\cap M_i| \geq 3d.
\end{aligned}$$
Thus, $\tau_{d}(P_m\square C_n)=3d$.
\end{proof}



\begin{Lem}
For even $m\geq 4$ and even $n\geq 4$, 
$\tau_{d}(P_m\square C_n)=4d-n$ if $n+1\leq d\leq \frac{mn}{2}$.
\end{Lem}
\begin{proof}
We construct a set $S$ consisting of the vertex subset 
$\{x_{11}, x_{13},\dots, x_{1,{n-1}}, x_{m2}, x_{m4},\\ \dots, x_{mn}\}$ 
and $d-n$ vertices of degree $4$ forming a stable set of $P_m\square C_n$. 
This is valid since $d-n\leq \frac{n(m-2)}{2}$. 
By Proposition 2.1, $\partial(S)$ is a $d$-transversal set with $|\partial(S)|=3n+4(d-n)=4d-n$. 
It is minimum since we construct four perfect matchings $M_1$, $M_2$, $M_3$, $M_4$ satisfying 
$$\begin{aligned}
 M_1\cap M_2=M_1\cap M_3=M_2\cap M_3=M_2\cap M_4=M_3\cap M_4=\emptyset
\end{aligned}$$
and $|M_1\cap M_4|=n$, defined as follows: 
$M_1=\{h_{i1}, h_{i3}, \dots, h_{i,n-1}| i=1,2,\dots, m\}$, 
$M_2=\{h_{i2}, h_{i4}, \dots, h_{in}| i=1,2,\dots, m\}$, 
$M_3=\{v_{1j}, v_{3j}, \dots, v_{m-1,j}| j=1,2,\dots, n\}$, 
$M_4=\{v_{2j}, v_{4j}, \dots, v_{m-2,j}| j=1,2,\dots, n\}\cup \{h_{i1}, h_{i3}, \dots, h_{i,n-1}| i=1,m\}$. 
These matchings are indicated by labels 1, 2, 3 and 4 in Fig.\hspace{0.06cm}\ref{D}, respectively. 
Any $d$-transversal set $T$ must satisfy $|T\cap M_i|\geq d$ for $1\leq i\leq 3$ and 
$|T\cap (M_4\backslash M_1)|\geq d-n$. 
Hence, $\tau_{d}(P_m\square C_n)=4d-n$.
\end{proof}

\begin{figure}[!htbp]
  \centering
  \subfigure{
  \begin{minipage}{6cm}
  \centering
  \includegraphics[totalheight=3.5cm]{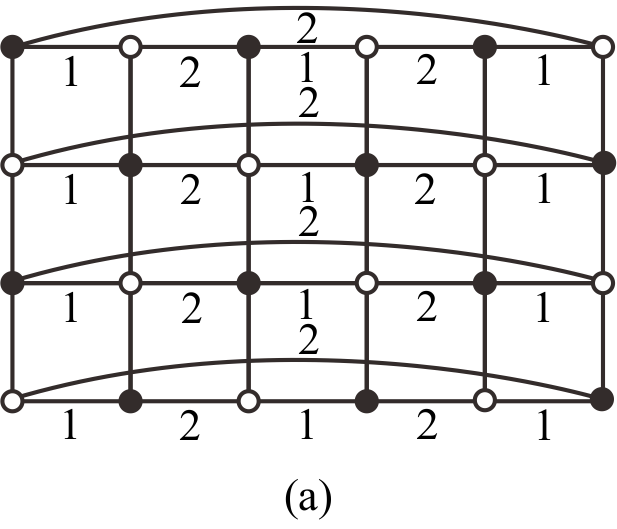}
   \end{minipage}%
   }
  \hspace{0.5cm}
  \subfigure{
  \begin{minipage}{6cm}
  \centering
  \includegraphics[totalheight=3.5cm]{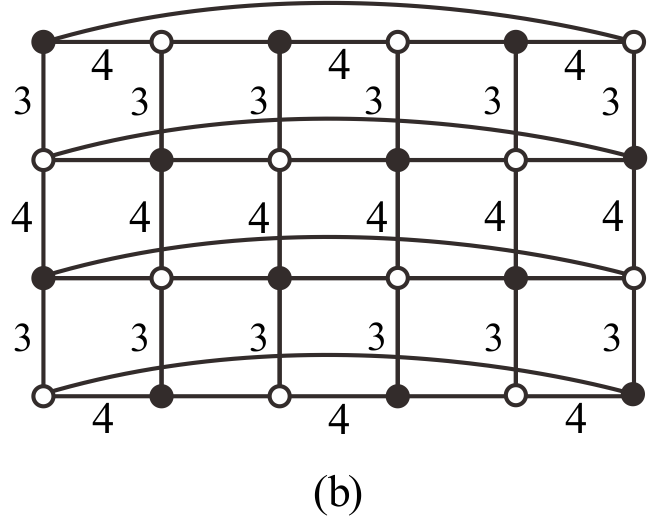}
   \end{minipage}
   }
  \caption{\label{D}\small{(a)Perfect matchings $M_1$ and $M_2$ of $P_4\square C_6$,
   (b)Perfect matchings $M_3$ and $M_4$ of $P_4\square C_6$.}}
\end{figure}

From Lemmas 3.1 and 3.2, we obtain the following result.

\begin{The}
Let $P_m\square C_n$ be a cylindrical grid with $m$ and $n$ even. Then\\
(a) for $m\geq2$ and $n\geq 4$, $\tau_{d}(P_m\square C_n)=3d$ if $1\leq d\leq n$,\\
(b) for $m\geq 4$ and $n\geq 4$, $\tau_{d}(P_m\square C_n)=4d-n$ if $n+1\leq d\leq \frac{mn}{2}$.\\
\end{The}

\subsection{$m$ is odd, $n$ is even}

The technique and construction in this case are similar to that of Subsection 3.1.

\begin{Lem}\rm{\cite{ZRB2009}}
For even $n\geq 4$, 
$\tau_{d}(C_n)=2d$ if $1\leq d\leq \frac{n}{2}$.
\end{Lem}

\begin{Lem}
For odd $m\geq 3$ and even $n\geq 4$, 
$\tau_{d}(P_m\square C_n)=3d$ if $1\leq d\leq n$.
\end{Lem}
\begin{proof}
Let $S=\{x_{11}, x_{13},\dots, x_{1,n-1}, x_{m1}, x_{m3}, \dots, x_{m,n-1}\}$ 
be a stable set of $P_m\square C_n$. By Proposition 2.1, 
we construct a $d$-transversal set $T$ of $P_m\square C_n$ by taking the edge subset $\partial(S')$, 
where $S'\subseteq S$ and $|S'|=d$. This is valid since $d\leq n$. 
Thus, $\tau_{d}(P_m\square C_n)\leq |T|=3d$. 
Furthermore, $P_m\square C_n$ contains three disjoint perfect matchings $M_1$, $M_2$, $M_3$, defined as follows: 
$M_1=\{h_{1j}| j=1,3,\dots, n-1\}\cup \{v_{2j}, v_{4j}, \dots, v_{m-1,j}| j=1,2,\dots, n\}$, 
$M_2=\{h_{mj}| j=1,3,\dots, n-1\}\cup \{v_{1j}, v_{3j}, \dots, v_{m-2,j}| j=1,2,\dots, n\}$, 
$M_3=\{h_{i2}, h_{i4}, \dots, h_{in}| i=1,2,\dots, m\}$. 
For any $d$-transversal set $T'$, we have
$$\begin{aligned}
 \tau_{d}(P_m\square C_n)\geq \sum\limits_{i=1}\limits^{3}|T'\cap M_i|\geq 3d.
\end{aligned}$$
Therefore, $\tau_{d}(P_m\square C_n)=3d$.
\end{proof}

\begin{Lem}
For odd $m\geq 3$ and even $n\geq 4$, 
$\tau_{d}(P_m\square C_n)=4d-n$ if $n+1\leq d\leq \frac{mn}{2}$.
\end{Lem}
\begin{proof}
Let $S=\{x_{22}, x_{24},\dots, x_{2n}, \dots, x_{{m-1},2}, x_{{m-1},4},\dots, x_{{m-1},n}\}$ 
be a stable set of $P_m\square C_n$. 
We construct a set $T$ by taking the edge set 
$\partial(\{x_{11}, x_{13}, \dots, x_{1,{n-1}}, x_{m1}, x_{m3}, \dots, x_{m,n-1} \})$ 
together with $\partial(S')$, where $S'\subseteq S$ and $|S'|=d-n$. 
These subsets are edge disjoint. This gives $|T|=3n+4(d-n)=4d-n$. 
By Proposition 2.1, $T$ is a $d$-transversal set. 
So $\tau_{d}(P_m\square C_n)\leq 4d-n$. 
Next, we construct four perfect matchings $M_1$, $M_2$, $M_3$, $M_4$ satisfying: 
$M_1\cap M_2=M_1\cap M_3=M_2\cap M_3=M_3\cap M_4=\emptyset$, $|M_1\cap M_4|=|M_2\cap M_4|=\frac{n}{2}$, 
defined as follows: 
$M_1=\{h_{1j}| j=1,3,\dots, n-1\}\cup \{v_{2j}, v_{4j}, \dots, v_{m-1,j}| j=1,2,\dots, n\}$, 
$M_2=\{h_{mj}| j=1,3,\dots, n-1\}\cup \{v_{1j}, v_{3j}, \dots, v_{m-2,j}| j=1,2,\dots, n\}$, 
$M_3=\{h_{i2}, h_{i4}, \dots, h_{in}| i=1,2,\dots, m\}$, 
$M_4=\{h_{i1}, h_{i3}, \dots, h_{i,n-1}| i=1,2,\dots, m\}$. 
Any $d$-transversal set $T'$ must satisfy $|T'\cap M_i|\geq d$ for $1\leq i\leq 3$, 
and $|T'\cap \{M_4\backslash (M_1\cup M_2)\}|\geq d-n$. 
Therefore, $\tau_{d}(P_m\square C_n)\geq 3d+(d-n)=4d-n$.
\end{proof}

By Lemmas 3.4 to 3.6, we have the following result.

\begin{The}
Let $P_m\square C_n$ be a cylindrical grid with $m$ odd and $n$ even. Then\\
(a) for $n\geq 4$, $\tau_{d}(C_n)=2d$ if $1\leq d\leq \frac{n}{2}$,\\
(b) for $m\geq 3$ and $n\geq 4$, $\tau_{d}(P_m\square C_n)=3d$ if $1\leq d\leq n$,\\
(c) for $m\geq 3$ and $n\geq 4$, $\tau_{d}(P_m\square C_n)=4d-n$ if $n+1\leq d\leq \frac{mn}{2}$.\\
\end{The}

\subsection{$m$ is even, $n$ is odd}

In this case, 
the graph $P_m\square C_n$ is non-bipartite and $\alpha(P_m\square C_n)=\frac{m(n-1)}{2}$ \cite{HK1996}.

\begin{Lem}
For even $m\geq 2$ and odd $n\geq 3$, 
$\tau_{d}(P_m\square C_n)=3d$ if $1\leq d\leq n-1$ and $\tau_{\frac{mn}{2}}(P_m\square C_n)=2mn-n$.
\end{Lem}
\begin{proof}
Let $S=\{x_{11}, x_{13},\dots, x_{1,n-2}, x_{m2}, x_{m4}, \dots, x_{m,n-1}\}$ 
be a stable set of $P_m\square C_n$. 
By Proposition 2.1, we construct a $d$-transversal set $T$ of $P_m\square C_n$ 
by taking the edge subset $\partial(S')$, where $S'\subseteq S$ and $|S'|=d$. 
This is valid for $d\leq |S|=n-1$. 
Thus, $\tau_{d}(P_m\square C_n)\leq |T|=3d$. 
Since $P_m\square C_n$ has three disjoint perfect matchings $M_1$, $M_2$, $M_3$, defined as:
$M_1=\{h_{i1}, h_{i3}, \dots, h_{i,n-2}| i=1,2,\dots, m\}\cup \{v_{1n}, v_{3n},\dots, v_{m-1,n}\}$, 
$M_2=\{h_{i2}, h_{i4}, \dots, h_{i,n-1}| i=1,2,\dots, m\}\cup \{v_{11}, v_{31},\dots, v_{m-1,1}\}$, 
$M_3=\{h_{in}| i=1,2,\dots, m\}\cup \{v_{1j}, v_{3j}, \dots, v_{m-1,j}| j=2,3,\dots, n-1\}$. 
Any $d$-transversal set $T'$ must satisfy $|T'\cap M_i|\geq d$ for $1\leq i\leq 3$. Therefore,
$$\begin{aligned}
 \tau_{d}(P_m\square C_n)\geq \sum\limits_{i=1}\limits^{3}|T'\cap M_i|\geq 3d.
\end{aligned}$$ 
Thus, $T$ is a minimum $d$-transversal set for $d\leq n-1$.

When $d=\frac{mn}{2}$, 
since $\nu(P_m\square C_n)=\frac{mn}{2}$ and $F_0(P_m\square C_n)=\emptyset$, 
a $\frac{mn}{2}$-transversal set of $P_m\square C_n$ must contain all $2mn-n$ edges. 
Hence, $\tau_{\frac{mn}{2}}(P_m\square C_n)=2mn-n$.
\end{proof}

In the following, 
we give upper and lower bounds for $\tau_{d}(P_m\square C_n)$ when $n\leq d< \frac{mn}{2}$.

\begin{Lem}
For even $m\geq 4$ and odd $n\geq 3$, 
$\tau_{d}(P_m\square C_n)\geq 4d-n-\lfloor \frac{d-n}{n-2}\rfloor$ if $n \leq d< \frac{mn}{2}$.
\end{Lem}
\begin{proof}
We partition the edge set of $P_m\square C_n$ into four matchings 
$M_1$, $M_2$, $M_3$ (as in Lemma 3.8) and 
$M_4=\{v_{2j}, v_{4j}, \dots, v_{m-2,j}| j=1,2,\dots, n\}$. 
Here, $|M_i|=\frac{mn}{2}$ for $1\leq i\leq 3$, and $|M_4|=\frac{mn}{2}-n$. 
Examples of these matchings are indicated by 1, 2, 3 and 4 in Fig.\hspace{0.06cm}\ref{E}, respectively. 
Let $T$ be a $d$-transversal set. Then $|T\cap M_i|\geq d$ for $1\leq i\leq 3$. 
Define $x_1, x_2, x_3\geq 0$ such that $|T\cap M_i|=d+x_i$. Thus,
$$\begin{aligned}
 |T\cap M_4|=|T|-3d-x_1-x_2-x_3.
\end{aligned}$$
For $j\in\{2,3,\dots,n-1\}$, 
construct a perfect matching $M^j_4$ (see Fig.\hspace{0.06cm}\ref{F}): 
$M^j_4=M_4\cup \{h_{1k}, h_{mk}| k=\dots,j-4, j-2, j+1, j+3,\dots \text{(mod n)}\}\cup 
\{v_{ij}| i=1,3,\dots, m-1\}\backslash\{v_{ij}| i=2,4,\dots, m-2\}$.
Observe that $T$ contains at least
$$\begin{aligned}
 |T\cap M^2_4|-|T\cap M_4|-(n-3)
  &\geq d-|T\cap M_4|-(n-3)\\
  &=4d-|T|+x_1+x_2+x_3-n+3
 \end{aligned}$$
edges of $M_3$ in column $2$ and horizontal edge set $\{h_{1n}$, $h_{mn}\}$. 
$T$ contains at least
$$\begin{aligned}
 |T\cap M^j_4|-|T\cap M_4|-(n-1)&\geq d-|T\cap M_4|-(n-1)\\&=4d-|T|+x_1+x_2+x_3-n+1
 \end{aligned}$$
edges of $M_3$ in column $j$ for $j=3,4,\dots, n-1$.\\
Combining all columns, we derive:
$$\begin{aligned}
 |T\cap M_3|&=d+x_3\\
 &\geq (4d-|T|+x_1+x_2+x_3-n+3)+(n-3)(4d-|T|+x_1+x_2+x_3-n+1)\\
 &=(n-2)(4d-|T|+x_1+x_2+x_3-n+1)+2.
 \end{aligned}$$
Then we have
$$\begin{aligned}
   \frac{d+x_3-2}{n-2}\geq 4d-|T|+x_1+x_2+x_3-n+1,
 \end{aligned}$$
hence,
$$\begin{aligned}
 |T|\geq 4d-n-\frac{d-n}{n-2}+x_1+x_2+\frac{n-3}{n-2}x_3.
 \end{aligned}$$
Since $x_1, x_2\geq 0$ and $\frac{n-3}{n-2}x_3\geq 0$.
$$\begin{aligned}
 |T|\geq 4d-n-\lfloor \frac{d-n}{n-2}\rfloor.
 \end{aligned} $$
The proof is complete.
\end{proof}

\begin{figure}[!htbp]
  \centering
  \subfigure{
  \begin{minipage}{6cm}
  \centering
  \includegraphics[totalheight=3.8cm]{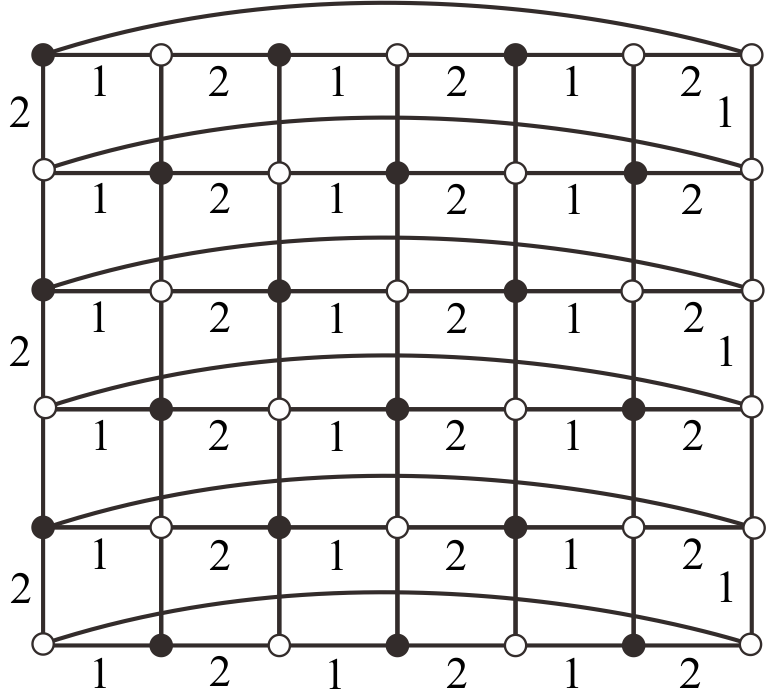}
   \end{minipage}%
   }%
  \hspace{0.5cm}
  \subfigure{
  \begin{minipage}{6cm}
  \centering
  \includegraphics[totalheight=3.5cm]{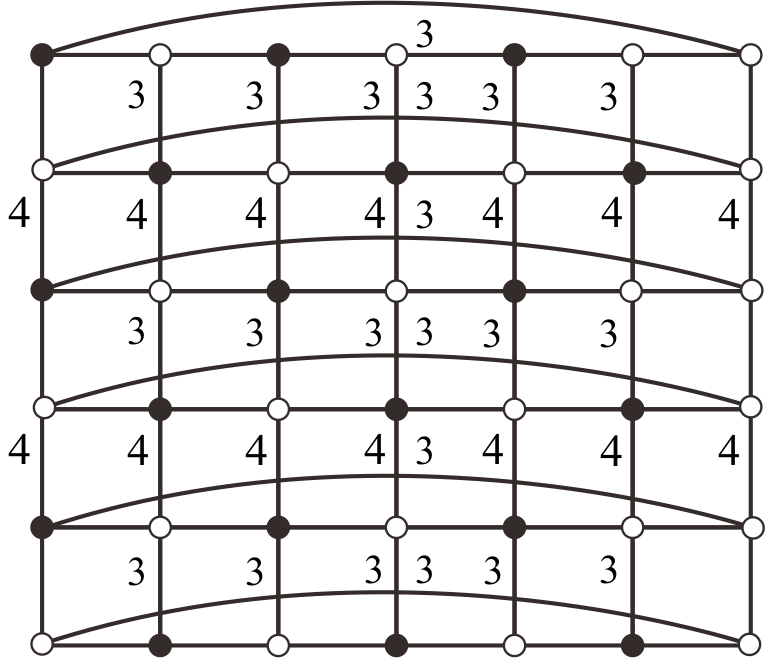}
   \end{minipage}
   }
  \caption{\label{E}\small{4-partition of $P_6\square C_7$.}}
\end{figure}

\begin{figure}[!htbp]
  \centering
  \subfigure{
  \begin{minipage}{6cm}
  \centering
  \includegraphics[totalheight=3.5cm]{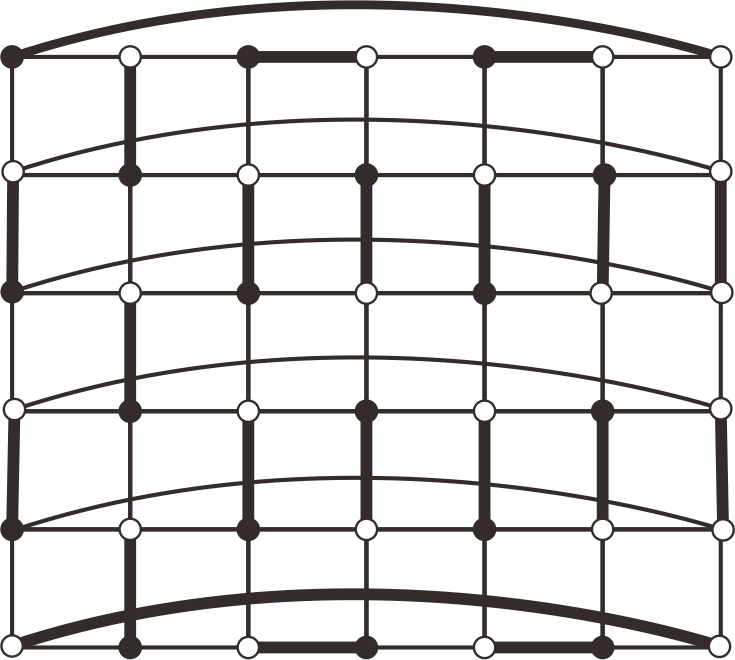}
   \end{minipage}%
   }%
  \hspace{0.5cm}
  \subfigure{
  \begin{minipage}{6cm}
  \centering
  \includegraphics[totalheight=3.5cm]{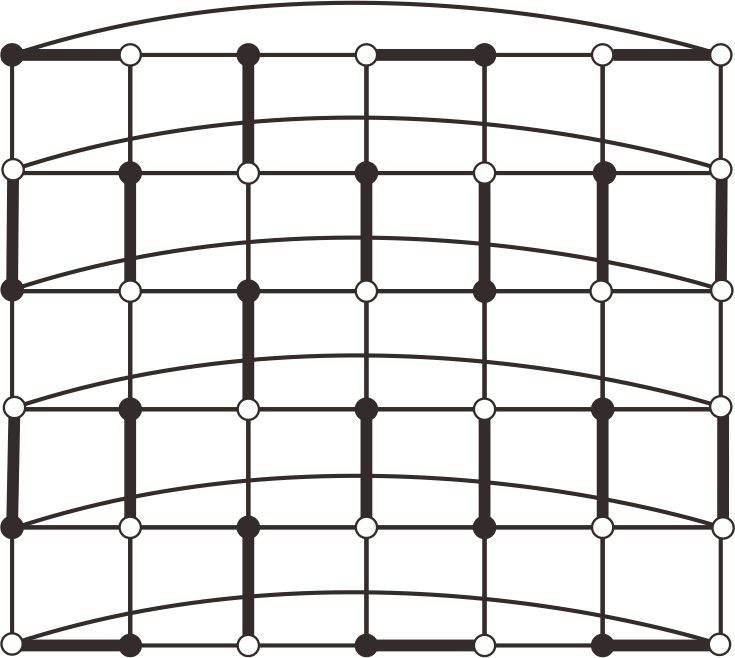}
   \end{minipage}
   }
  \caption{\label{F}\small{Perfect matchings $M^2_4$ (left) and $M^3_4$ (right) of $P_6\square C_7$.}}
\end{figure}


\begin{Lem}
For even $m\geq 4$ and odd $n\geq 3$, 
$\tau_{d}(P_m\square C_n)\geq 4d-n-\lfloor\frac{mn-2d}{2n-2}\rfloor$ if $d< \frac{mn}{2}$.
\end{Lem}
\begin{proof}
Let $k=\lfloor\frac{mn-2d}{2n-2}\rfloor$. 
We show that no $d$-transversal set $T$ can have at most $4d-n-k-1$ edges. 
Equivalently, for any $E\subseteq E(P_m\square C_n)$ with
$$\begin{aligned}
 |E|\geq 2mn-n-(4d-n-k-1)=4(\frac{mn}{2}-d)+k+1,
\end{aligned}$$ 
there exits a perfect matching $M$ of $P_m\square C_n$ such that $|E\cap M|\geq \frac{mn}{2}-d+1$. 
If $|E\cap M_i|\geq \frac{mn}{2}-d+1$ for $1\leq i\leq 3$, 
the result follows since $M_1$, $M_2$ and $M_3$ are perfect matchings. 
So, suppose $M_4$ contains at least $\frac{mn}{2}-d+k+1$ edges of $E$. 
The edges of $M_4$ lie in columns $j (1\leq j\leq n)$. 
We claim that there exists $j'$ such that the union of all columns $j\neq j'$ 
contains at least $\frac{mn}{2}-d+1$ edges of $E\cap M_4$.

Let $e_j$ denote the number of edges of $E\cap M_4$ in column $j$. 
Assume, for contradiction, that for every $j'$,
$$\begin{aligned}
\sum\limits_{j\neq j'}e_j\leq \frac{mn}{2}-d.
\end{aligned}$$
Summing over all $j'$, we get
$$\begin{aligned}
\sum\limits_{ j'=1}\limits^{n}\sum\limits_{j\neq j'}e_j\leq \sum\limits_{ j'=1}\limits^{n}(\frac{mn}{2}-d),
\end{aligned}$$
implying
$$\begin{aligned}
\sum\limits_{j=1}\limits^{n}e_j&\leq \frac{n}{n-1}(\frac{mn}{2}-d)<\frac{mn}{2}-d+k+1,
\end{aligned}$$
where the last inequality follows from $k>\frac{mn-2d}{2n-2}-1$. 
This is a contradiction. Hence, 
there is a perfect matching $M^j_4$ containing at least $\frac{mn}{2}-d+1$ edges of $E\cap M_4$. 
From the argument, we conclude that $\tau_{d}(P_m\square C_n)\geq 4d-n-\lfloor\frac{mn-2d}{2n-2}\rfloor$.
\end{proof}

\begin{Lem}
For even $m\geq 4$ and odd $n\geq 3$, 
$\tau_{d}(P_m\square C_n)\leq 4d-n+1$ if $n\leq d\leq n+\frac{m}{2}-2$, 
and $\tau_{d}(P_m\square C_n)\leq 4d-n$ if $n+\frac{m}{2}-1\leq d< \frac{mn}{2}$.
\end{Lem}
\begin{proof}
Let $S=\{x_{22}, x_{24},\dots, x_{2,n-1}, \dots, x_{m-1,1}, x_{m-1,3}, \dots, x_{m-1,n-2}\}$ 
be a stable set of $P_m\square C_n$. 
we construct a set $T$ of $P_m\square C_n$ by taking the edge subset 
$\partial(\{x_{11}, x_{13}, \dots, x_{1,{n-2}},\\ x_{m2}, x_{m4}, \dots, x_{m,n-1}\})$ 
together with $\partial(S')$, where $S'\subseteq S$ and $|S'|=d-(n-1)$. 
These edge sets are disjoint. 
This is valid since $d-(n-1)\leq |S|=(m-2)\frac{n-1}{2}$ for $d\leq n+\frac{m}{2}-2$. 
By Proposition 2.1, $T$ is a $d$-transversal set with
$$\begin{aligned}
|T|=3(n-1)+4(d-n+1)=4d-n+1.
\end{aligned}$$
Thus, $\tau_{d}(P_m\square C_n)\leq 4d-n+1$.

For $n+\frac{m}{2}-1\leq d\leq \frac{mn}{2}$, 
let $S''\subseteq S$ and $|S''|=\frac{mn}{2}-d$. 
Let $G$ be the spanning subgraph of $P_m\square C_n$ satisfying $E(G)=\partial(S'')$. 
This is valid since $\frac{mn}{2}-d\leq (m-2)\frac{n-1}{2}$ for $d\geq n+\frac{m}{2}-1$. 
As $P_m\square C_n$ has perfect matchings, 
$\nu(G)=\frac{mn}{2}-d$. 
Let $T'=E(P_m\square C_n)\backslash \partial(S'')$. 
For any perfect matching $M$, $|T'\cap M|\geq d$.
Thus, $T'$ is a $d$-transversal set with
$$\begin{aligned}
|T'|=2mn-n-4(\frac{mn}{2}-d)=4d-n.
\end{aligned}$$
Hence, $\tau_{d}(P_m\square C_n)\leq 4d-n$.
\end{proof}

From Lemmas 3.8 to 3.11, we obtain the following result.

\begin{The}
Let $P_m\square C_n$ be a cylindrical grid with $m$ even and $n$ odd. Then\\
1. for $n\geq 3$, $\tau_{d}(P_2\square C_n)=3d$ if $1\leq d\leq n$.\\
2. for $m\geq 4$ and $n\geq 3$,\\
$(a) \hspace{0.1cm}\tau_{d}(P_m\square C_n)=3d$ if $1\leq d\leq n-1$,\\
$(b) \hspace{0.1cm}4d-n-min\{\lfloor \frac{d-n}{n-2}\rfloor, \lfloor\frac{mn-2d}{2n-2}\rfloor\} \leq \tau_{d}(P_m\square C_n)\leq 4d-n+1$ if $n\leq d\leq n+\frac{m}{2}-2$,\\
$(c) \hspace{0.1cm}4d-n-min\{\lfloor \frac{d-n}{n-2}\rfloor, \lfloor\frac{mn-2d}{2n-2}\rfloor\} \leq \tau_{d}(P_m\square C_n)\leq 4d-n$ if $n+\frac{m}{2}-1\leq d< \frac{mn}{2}$,\\
$(d) \hspace{0.1cm}\tau_{\frac{mn}{2}}(P_m\square C_n)=2mn-n$.
\end{The}

For $n=3$, we can obtain an explicit formula for $\tau_{d}(P_m\square C_3)$.

\begin{Lem}
For even $m\geq4$, 
$\tau_{d}(P_m\square C_3)=3d$ if $1\leq d\leq \frac{m}{2}+2$.
\end{Lem}
\begin{proof}
Let $M$ be a perfect matching of $P_m\square C_3$. 
For each odd integer $i$, $1\leq i\leq m-1$, 
we show that $|\{v_{i1},v_{i2},v_{i3}\}\cap M|\geq 1$. 
By contradiction, suppose there exists an odd $i$ such that $|\{v_{i1},v_{i2},v_{i3}\}\cap M|= 0$. 
Then the subgraph induced by the vertex set $X=\{x_{kj}|1\leq k\leq i, 1\leq j\leq 3\}$ 
must contain a perfect matching. This is impossible since $|X|=3i$ is odd.

For $1\leq d\leq \frac{m}{2}$, 
let $T=\{v_{2i-1,1}, v_{2i-1,2}, v_{2i-1,3}|1\leq i\leq d\}$. 
From the above discussion, 
for any perfect matching $M$, we have $|T\cap M|\geq d$.

If $d=\frac{m}{2}+1$, 
let $T=\{v_{2i-1,1}, v_{2i-1,2}, v_{2i-1,3}|1\leq i\leq \frac{m}{2}\}\cup \{h_{11}, h_{12}, h_{13}\}$. 
For any perfect matching $M$, if $|\{v_{11},v_{12},v_{13}\}\cap M|= 1$ (resp. $3$), 
then $|\{h_{11}, h_{12}, h_{13}\}\cap M|= 1$ (resp. $0$). 
Thus $|\{v_{11},v_{12},v_{13}, h_{11}, h_{12}, h_{13}\}\cap M|\geq 2$, 
and we obtain $|T\cap M|\geq \frac{m}{2}+1$.

If $d=\frac{m}{2}+2$, let 
$T=\{v_{2i-1,1}, v_{2i-1,2}, v_{2i-1,3}|1\leq i\leq \frac{m}{2}\}\cup \{h_{11}, h_{12}, h_{13}, h_{m1}, h_{m2}, h_{m3}\}$. 
Similarly, for any perfect matching $M$, 
we have $|\{v_{m-1,1},v_{m-1,2},v_{m-1,3}, h_{m1}, h_{m2}, h_{m3}\}\cap M|\geq 2$, 
hence $|T\cap M|\geq \frac{m}{2}+2$. 
Therefore, $T$ is a $d$-transversal set with $|T|=3d$.

Since $P_m\square C_3$ has three disjoint perfect matchings $M_1$, $M_2$, $M_3$ (see Fig.\hspace{0.06cm}\ref{G}), 
where $M_1$, $M_2$ and $M_3$ are defined as in Lemma 3.8, 
any $d$-transversal set must contain at least $3d$ edges. 
Thus $\tau_{d}(P_m\square C_3)=3d$.
\end{proof}

\begin{figure}[!htbp]
\begin{center}
\includegraphics[totalheight=3.2cm]{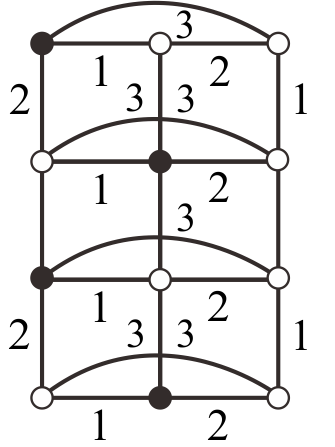}
\caption{\label{G}\small{Perfect matchings $M_1$, $M_2$ and $M_3$ of $P_4\square C_3$.}}
\end{center}
\end{figure}

\begin{Lem}
For even $m\geq4$, 
$\tau_{d}(P_m\square C_3)=4d-\frac{m}{2}-2+\lfloor \frac{2d-m-2}{4}\rfloor$ if $\frac{m}{2}+3\leq d\leq \frac{3m}{2}$.
\end{Lem}
\begin{proof}
Let $t=\lfloor \frac{2d-m-2}{4}\rfloor$. 
We show that no $d$-transversal set $T$ of $P_m\square C_3$ can have at most $4d-\frac{m}{2}-3+t$ edges. 
It suffices to prove that for any $E\subseteq E(P_m\square C_3)$ with
$$\begin{aligned}
|E|\geq 6m-3-(4d-\frac{m}{2}-3+t)=4(\frac{3m}{2}-d)+\frac{m}{2}-t,
\end{aligned}$$
there exists a perfect matching $M$ of $P_m\square C_3$ such that $|E\cap M|\geq \frac{3m}{2}-d+1$. 
If $|E\cap M_i|\geq \frac{3m}{2}-d+1$ for $1\leq i\leq 3$, 
then we are done since $M_1$, $M_2$ and $M_3$ are perfect matchings. 
Hence we may assume that $M_4$ contains at least $2m-d-t$ edges of $E$, 
where $M_1$, $M_2$, $M_3$ and $M_4$ are defined as in Lemma 3.9. 
The edge set of $M_4$ is contained in $M^j_4$ except for the column $j$, $1\leq j\leq 3$.

We claim that there exists a column $j'$ such that the union of all columns 
$j\neq j'$ contains at least $\frac{3m}{2}-d+1$ edges of $E\cap M_4$. 
Let $e_j$ be the number of edges of $E\cap M_4$ in column $j$. 
Suppose, to the contrary, that
$$\begin{aligned}
\sum\limits_{j\neq j'}e_j\leq \frac{3m}{2}-d
\end{aligned}$$
for all $j'$. Summing this inequality over all columns, we deduce
$$\begin{aligned}
\sum\limits_{ j'=1}\limits^{3}\sum\limits_{j\neq j'}e_j\leq \sum\limits_{j'=1}\limits^{3}(\frac{3m}{2}-d),
\end{aligned}$$
which simplifies to
$$\begin{aligned}
|E\cap M_4|=\sum\limits_{ j=1}\limits^{3}e_j&\leq \frac{3}{2}(\frac{3m}{2}-d)<2m-d-t.
\end{aligned}$$
Here the last inequality follows from $\frac{2d-m}{4}> t$, contradicting $|E\cap M_4|\geq 2m-d-t$. 
Hence $\tau_{d}(P_m\square C_3)\geq 4d-\frac{m}{2}-2+\lfloor \frac{2d-m-2}{4}\rfloor$.

Now we construct such a minimum $d$-transversal set $T$. 
For $\frac{m}{2}+3\leq d\leq \frac{3m}{2}$, take the edge subset 
$S=\{v_{2i-1,1}, v_{2i-1,2}, v_{2i-1,3}|1\leq i\leq \frac{m}{2}\}\cup \{h_{11}, h_{12}, h_{13}, h_{m1}, h_{m2}, h_{m3}\}$. 
Next, we add sets $T_2\cup T_3\cup \dots \cup T_{d-\frac{m}{2}-1}$, 
where $T_i=\{h_{i1}, h_{i2}, h_{i3}, v_{i1}, v_{i3}\}$ for even $2\leq i\leq m-2$ and 
$T_i=\{h_{i1}, h_{i2}, h_{i3}, v_{i-1,2}\}$ for odd $3\leq i\leq m-1$. 
Let $T=S\cup T_2\cup T_3\cup \dots \cup T_{d-\frac{m}{2}-1}$. 
By an argument similar to Lemma 3.13, 
we have $|T_i\cap M|\leq 1$ for odd $i$ and $|(T_i\cup T_{i+1})\cap M|\leq 2$ for even $i$, 
where $M$ is any perfect matching of $P_m\square C_3$. Consequently,
$$\begin{aligned}
 |(\{T_2\cup T_3\cup \dots \cup T_{m-1}\} \backslash \{T_2\cup T_3\cup \dots \cup T_{d-\frac{m}{2}-1}\})\cap M|&\leq m-1-(d-\frac{m}{2})+1\\
 &=\frac{3m}{2}-d.
\end{aligned}$$
Thus, $|T\cap M|\geq \frac{3m}{2}-(\frac{3m}{2}-d)=d$ for every perfect matching $M$. 
This implies $T$ is a $d$-transversal set of $P_m\square C_3$. For its cardinality,
$$\begin{aligned}
|T|=3(\frac{m}{2}+2)+4(d-\frac{m}{2}-2)+\lfloor \frac{2d-m-2}{4}\rfloor=4d-\frac{m}{2}-2+\lfloor \frac{2d-m-2}{4}\rfloor.
\end{aligned}$$
The proof is completed.
\end{proof}

The lower bounds in Theorem 3.12 are sharp. 
To see this, consider the graph $G=P_6\square C_3$ for $3\leq d\leq 7$. 
From Theorem 3.12 2.(b), we have $9\leq \tau_{3}(G)\leq 10$ and $12\leq \tau_{4}(G)\leq 14$. 
From Theorem 3.12 2.(c), it follows that $15\leq \tau_{5}(G)\leq 17$, $20\leq \tau_{6}(G)\leq 21$ 
and $24\leq \tau_{7}(G)\leq 25$. However, by applying Lemmas 3.13 and 3.14, 
we obtain the exact values $\tau_{3}(G)=9$, $\tau_{4}(G)=12$, $\tau_{5}(G)=15$, 
$\tau_{6}(G)=20$ and $\tau_{7}(G)=24$, respectively.

\subsection{$m$ and $n$ are odd}

In this case, $P_m\square C_n$ is non-bipartite and admits an almost perfect matching. 
Moreover, for every vertex $v$, there exists an almost perfect matching that covers all vertices except $v$. 
Consequently, we have $S(P_m\square C_n)=F_0(P_m\square C_n)=F_1(P_m\square C_n)=\emptyset$.

The following result was first established in \cite{ZRB2009}. 
However, the authors did not provide a proof and left it to the reader. 
Here, we present a proof.

\begin{Lem}\rm{\cite{ZRB2009}}
For odd $n\geq 3$, 
$\tau_{d}(C_n)=2d+1$ if $1\leq d\leq \frac{n-1}{2}$.
\end{Lem}
\begin{proof}
For any $d$-transversal set $T$, 
we first show that $|T|\geq 2d+1$. 
Consider the following $n$ almost perfect matchings of $C_n$: 
$M_1=\{h_{12}, h_{14}, \dots, h_{1,n-1}\}$, 
$M_2=\{h_{1n}, h_{13}, \dots, h_{1,n-2}\}$, $\dots,$ $M_n=\{h_{11}, h_{13}, \dots, h_{1,n-2}\}$. 
Note that $x_{1i}$ is the only uncovered vertex by $M_i$, 
where $1\leq i\leq n$. For each edge $e$ of $C_n$, we have
$$\begin{aligned}
|\{M_i|e\in M_i; 1\leq i\leq n \}|=\frac{n-1}{2}.
\end{aligned}$$
Thus, for every $d$-transversal set $T$, 
$$\begin{aligned}
\frac{n-1}{2}|T|=\sum\limits_{e\in T}\sum\limits^n_{i=1}|\{e\}\cap M_i|=\sum\limits^n_{i=1}\sum\limits_{e\in T}|\{e\}\cap M_i|\geq nd,
\end{aligned}$$
implying $|T|\geq \frac{2nd}{n-1}>2d$.

For $1\leq d\leq \frac{n-1}{2}$, 
we construct a $d$-transversal set $T$ of $C_n$ by taking 
$T=\partial(x_{11})\cup \partial(x_{12})\cup \dots \cup \partial(x_{1,2d-1}) \cup \partial(x_{1,2d})$, 
which is valid since $2d\leq n-1$. 
Thus, $\tau_{d}(C_n)\leq |T|=2d+1$. 
From the above argument, we conclude that $\tau_{d}(C_n)=2d+1$.
\end{proof}

\begin{Lem}
For odd integers $m\geq 3$ and $n\geq 3$, 
$\tau_{d}(P_m\square C_n)=4d-n+2$ if $\frac{m-3}{2}+n\leq d\leq \frac{mn-1}{2}$.
\end{Lem}
\begin{proof}
Let $S=\{x_{22}, x_{24},\dots, x_{2,n-1},\dots ,x_{m-1,2}, x_{m-1,4}, \dots, x_{m-1,n-1}\}$ 
be a stable set of $P_m\square C_n$, 
and let $S'\subseteq S$ with $|S'|=\frac{mn-1}{2}-d$. 
Assume $\frac{m-3}{2}+n\leq d\leq \frac{mn-1}{2}$. 
Let $G$ be the spanning subgraph of $P_m\square C_n$ satisfying $E(G)=\partial(S')$. 
This is valid since $\frac{mn-1}{2}-d\leq |S|=(m-2)\frac{n-1}{2}$. 
Thus, $\nu(G)\leq \frac{mn-1}{2}-d$. 
Let $T=E(P_m\square C_n)\backslash \partial(S')$. 
Then $|T\cap M|\geq d$ for every almost perfect matching $M$ of $P_m\square C_n$. 
Therefore, $T$ is a $d$-transversal set with
$$\begin{aligned}
 |T|=2mn-n-4(\frac{mn-1}{2}-d)=4d-n+2.
\end{aligned}$$
This is minimum, as we construct four almost perfect matchings $M_1$, $M_2$, $M_3$, $M_4$ satisfying 
$|M_1\cap M_2|=|M_1\cap M_3|=|M_2\cap M_3|=|M_1\cap M_4|=0$, $|M_2\cap M_4|=\frac{n-1}{2}$ 
and $|M_3\cap M_4|=\frac{n-3}{2}$. Define 
$M_1=\{h_{i1}, h_{i3}, \dots, h_{i,n-2}| i=1, 2, \dots, m\}\cup \{v_{2n}, v_{4n}, \dots, v_{m-1,n}\}$, 
$M_2=\{v_{1j}, v_{3j}, \dots, v_{m-2,j}| j=1, 2, \dots, n\}\cup \{h_{m2}, h_{m4}, \dots, h_{m,n-1}\}$, 
$M_3=\{v_{2j}, v_{4j}, \dots, v_{m-1,j}| j=2, 3, \dots, n-1\}\cup \{h_{1n}, h_{2n}, \dots, h_{mn}\} 
\cup \{h_{12}, h_{14}, \dots, h_{1,n-3}\}$, 
$M_4=\{h_{i2}, v_{i4}, \dots, v_{i,n-1}| i=1, 2, \dots, n\}\cup \{v_{21}, v_{41}, \dots, h_{m-1,1}\}$. 
These are indicated by 1, 2, 3 and 4 in Fig.\hspace{0.06cm}\ref{H}, respectively. 
Any $d$-transversal set $T'$ must satisfy $|T'\cap M_i|\geq d$ for $1\leq i\leq 3$ and
$$\begin{aligned}
|T'\cap \{M_4\backslash (M_2\cup M_3)\}|\geq d-(\frac{n-1}{2}+\frac{n-3}{2})=d-n+2,
\end{aligned}$$
implying that $|T'|\geq 4d-n-2$.
Thus, $T$ is minimum.
\end{proof}

\begin{figure}[!htbp]
  \centering
  \subfigure{
  \begin{minipage}{6cm}
  \centering
  \includegraphics[totalheight=3.8cm]{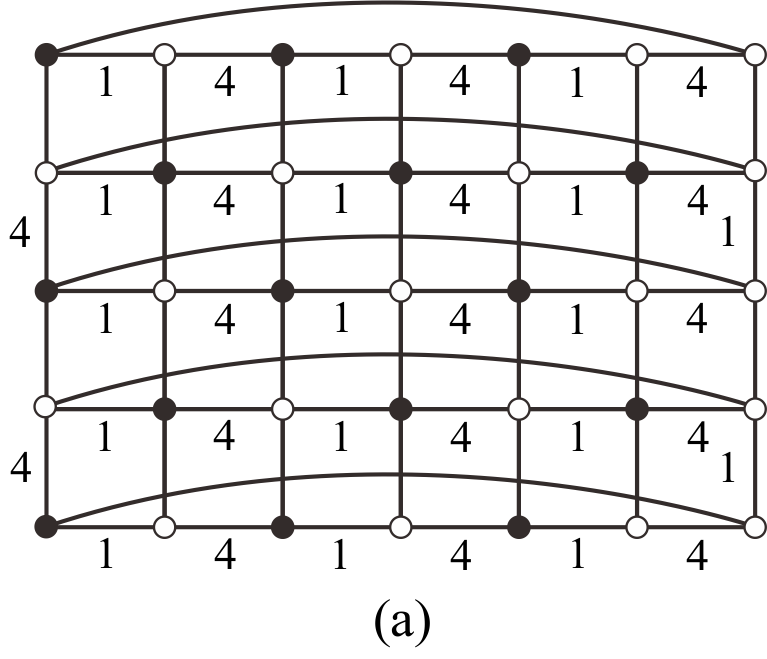}
   \end{minipage}%
   }%
  \hspace{0.5cm}
  \subfigure{
  \begin{minipage}{6cm}
  \centering
  \includegraphics[totalheight=3.8cm]{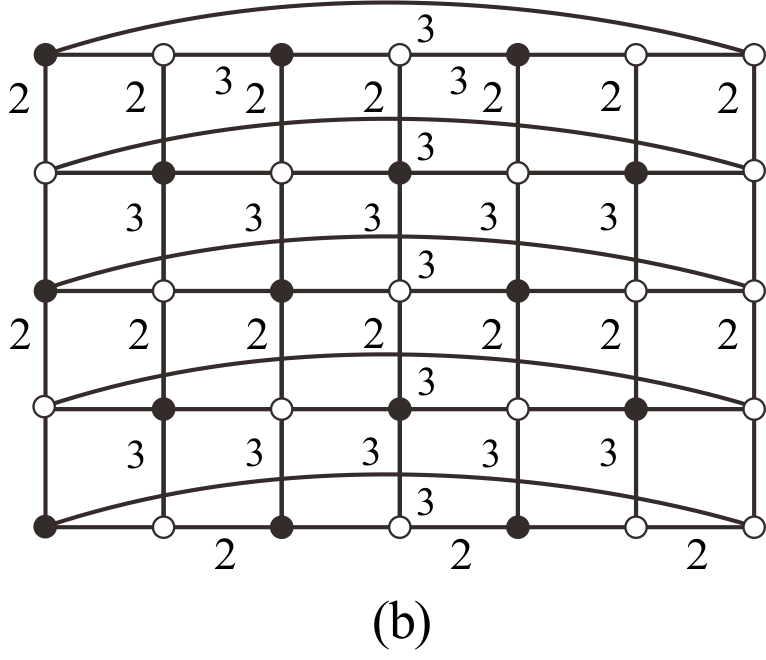}
   \end{minipage}
   }
  \caption{\label{H}\small{(a)$M_1$ and $M_4$ of $P_5\square C_7$, (b)$M_2$ and $M_3$ of $P_5\square C_7$.}}
\end{figure}

\begin{Lem}
For odd integers $m\geq 3$ and $n\geq 3$, 
$max\{3d, 4d-n+2\}\leq \tau_{d}(P_m\square C_n)\leq 3d+\frac{m+1}{2}$ if $\frac{m-1}{2}\leq d\leq \frac{m-5}{2}+n$.
\end{Lem}
\begin{proof}
Let $S=\{x_{11}, x_{13},\dots, x_{1,n-2}, x_{m1}, x_{m3}, \dots, x_{m,n-2}\}$ 
be a stable set of $P_m\square C_n$, 
and let $S'\subseteq S$ with $|S'|=\frac{m-3}{2}+n-d$. 
When $\frac{m-1}{2}\leq d\leq \frac{m-5}{2}+n$, 
define $G$ as the spanning subgraph of $P_m\square C_n$ satisfying
$$\begin{aligned}
E(G)=\partial(\{x_{22},x_{24},\dots, x_{2,{n-1}},\dots, x_{m-1,2},x_{m-1,4},\dots, x_{m-1,n-1}\})\cup \partial(S').
\end{aligned}$$
Then $\nu(G)\leq \frac{mn-1}{2}-d$. 
Let $T=E(P_m\square C_n)\backslash E(G)$. 
For every almost perfect matching $M$ of $P_m\square C_n$, 
we have $|T\cap M|\geq d$. 
Thus $T$ is a $d$-transversal set with
$$\begin{aligned}
|T|=2mn-n-4\frac{(m-2)(n-1)}{2}-3(\frac{m-3}{2}+n-d)=3d+\frac{m+1}{2}.
\end{aligned}$$
Hence, $\tau_{d}(P_m\square C_n)\leq 3d+\frac{m+1}{2}$.

Similarly to Lemma 3.16, 
we have $\tau_{d}(P_m\square C_n)\geq 4d-n+2$. 
Furthermore, any $d$-transversal set $T'$ must satisfy $|T'|\geq 3d$ 
because $P_m\square C_n$ contains three disjoint almost perfect matchings $M_1$, $M_2$ and $M_3$. 
Therefore, max$\{3d, 4d-n+2\}\leq \tau_{d}(P_m\square C_n)$.
\end{proof}

The upper bound in Lemma 3.17 is sharp. 
Consider the following example.

\noindent\textbf{Example.} Consider the graph $P_3\square C_3$ for $d=1$. 
By Lemma 3.17, we have $3\leq \tau_{1}(P_3\square C_3)\leq 5$. 
A set $T$ is a $1$-transversal set if and only if no almost perfect matching of 
$P_3\square C_3$ contains $4$ edges in $E(P_3\square C_3)\backslash T$. 
It can be verified that the minimum $1$-transversal set (illustrated in Fig.\hspace{0.06cm}\ref{I}) 
has size 5. Thus, $\tau_{1}(P_3\square C_3)=5$.

\begin{figure}[!htbp]
\begin{center}
\includegraphics[totalheight=2cm]{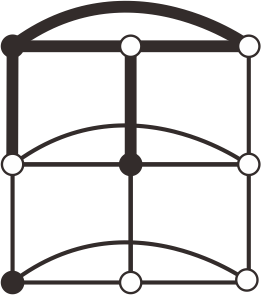}
\caption{\label{I}\small{A minimum $1$-transversal set (thick edges) of size $5$.}}
\end{center}
\end{figure}

A \textit{Hamiltonian path} (resp. cycle) of a graph $G$ is a spanning path (resp. cycle) that contains every vertex of $G$.

\begin{Lem}
For odd integers $m\geq 5$ and $n\geq 3$, 
$max\{3d, 4d-n+2\}\leq \tau_{d}(P_m\square C_n)\leq 4d+1$ if $1\leq d\leq \frac{m-3}{2}$.
\end{Lem}
\begin{proof}
Similarly to the proof of Lemma 3.17, 
we have $\tau_{d}(P_m\square C_n)\geq$ max$\{3d, 4d-n+2\}$. 
Let $P=u_1u_2\dots u_{mn}=x_{11}x_{12}\dots x_{1n}x_{2n}x_{2,n-1}\dots x_{21}x_{31}x_{32}\dots x_{3n} x_{4n}\dots x_{mn}$ 
denote a Hamilton path of $P_m\square C_n$. 
We construct a $d$-transversal set $T$ of $P_m\square C_n$ by taking the union of neighborhoods 
$T=\partial(u_1)\cup \partial(u_2)\cup \dots \cup\partial(u_{2d})$. 
This construction is valid since $2d\leq m-3<mn$. 
The size of $T$ satisfies $|T|\leq 4d+1$. 
Combining these results, 
we conclude max$\{3d, 4d-n+2\}\leq \tau_{d}(P_m\square C_n)\leq 4d+1$ for $1\leq d\leq \frac{m-3}{2}$.
\end{proof}

From Lemmas 3.15 to 3.18, we get the following result.

\begin{The}
Let $P_m\square C_n$ be a cylindrical grid with $m$ and $n$ odd. Then\\
1. for $n\geq 3$, $\tau_{d}(C_n)=2d+1$ if $1\leq d\leq \frac{n-1}{2}$.\\
2. for $m\geq 3$ and $n\geq 3$,\\
$(a) \hspace{0.1cm}max\{3d, 4d-n+2\}\leq \tau_{d}(P_m\square C_n)\leq 4d+1$ if $1\leq d\leq \frac{m-3}{2}$,\\
$(b) \hspace{0.1cm}max\{3d, 4d-n+2\}\leq \tau_{d}(P_m\square C_n)\leq 3d+\frac{m+1}{2}$ if $\frac{m-1}{2}\leq d\leq \frac{m-5}{2}+n$,\\
$(c) \hspace{0.1cm}\tau_{d}(P_m\square C_n)=4d-n+2$ if $\frac{m-3}{2}+n\leq d\leq \frac{mn-1}{2}$.
\end{The}

\section{$d$-transversal number of $C_m\square C_n$}

We first give the $d$-transversal number for even order toroidal grid $C_m\square C_n$.

\begin{Lem}
Let $m\geq 4$ and $n\geq 4$ be two even integers.
Then $\tau_{d}(C_m\square C_n)=4d$ for $1\leq d\leq \frac{mn}{2}$.
\end{Lem}
\begin{proof}
Since $C_m\square C_n$ is a 4-regular bipartite graph, 
the lemma follows immediately from Theorem 1.1 or Proposition 2.3.
\end{proof}

\begin{Lem}\rm{\cite{HK1996}}
Let $m\geq 3$ and odd $n\geq 3$. 
Then $\alpha(C_m\square C_n)=\frac{m(n-1)}{2}$.
\end{Lem}

\begin{Lem}\rm{\cite{ABS1990}}
Every toroidal grid $C_m\square C_n$ with $m\geq 3$ and $n\geq 3$ has two disjoint Hamiltonian cycles.
\end{Lem}

\begin{Lem}
For even $m\geq 4$ and odd $n\geq 3$, 
$\tau_{d}(C_m\square C_n)=4d$ for $1\leq d\leq \frac{mn}{2}$.
\end{Lem}
\begin{proof}
By Lemma 4.2, $\alpha(C_m\square C_n)= \frac{m(n-1)}{2} \geq \frac{mn}{4}+1> \lceil \frac{mn}{4}\rceil$ 
for even $m\geq 4$ and odd $n\geq 3$. 
By Lemma 4.3, $C_m\square C_n$ admits 4 disjoint perfect matchings.
From Proposition 2.3, it follows that $\tau_{d}(C_m\square C_n)=4d$ for $1\leq d\leq \frac{mn}{2}$.
\end{proof}

Next we consider odd order toroidal grid $C_m\square C_n$ with $m, n$ odd 
and at least 3. Without loss of generality, assume $m\geq n$.

\begin{Lem}
For odd integers $m\geq 3$ and $n\geq 3$, 
$\tau_{d}(C_m\square C_n)= 4d+2$ if $\frac{m-1}{2}\leq d\leq \frac{mn-1}{2}$.
\end{Lem}
\begin{proof}
Let $\frac{m-1}{2}\leq d\leq \frac{mn-1}{2}$. 
Let $G$ be a spanning subgraph of $C_m\square C_n$ satisfying $E(G)=\partial(S)$, 
where $S$ is a stable set of $C_m\square C_n$ with $|S|=\frac{mn-1}{2}-d$. 
By Lemma 4.2, this construction is valid since $\frac{mn-1}{2}-d \leq \frac{m(n-1)}{2}$. 
Thus $\nu(G)\leq \frac{mn-1}{2}-d$. Let $T=E(C_m\square C_n)\backslash \partial(S)$. 
Then, for every almost perfect matching $M$ of $C_m\square C_n$, we have $|T\cap M|\geq d$. 
Consequently, $T$ is a $d$-transversal set satisfying $|T|=2mn-4(\frac{mn-1}{2}-d)=4d+2$. 
Hence, $\tau_{d}(C_m\square C_n)\leq 4d+2$.

By Lemma 4.3, $C_m\square C_n$ contains two disjoint Hamiltonian cycles $\mathcal{C}_1$ and $\mathcal{C}_2$. 
Let $T'$ be any $d$-transversal set of $C_m\square C_n$. 
We claim that $T'\cap E(\mathcal{C}_i)$ is a $d$-transversal set of $\mathcal{C}_i$ for $i=1, 2$. 
If not, then there exists an almost perfect matching $M_i$ of $\mathcal{C}_i$ 
such that $|M_i\cap (T'\cap E(\mathcal{C}_i))|<d$. 
Since $M_i$ is also an almost perfect matching of $C_m\square C_n$, 
this implies $|M_i\cap T'|=|M_i\cap (T'\cap E(\mathcal{C}_i))|<d$, a contradiction. 
By Lemma 3.15, we have $|T'\cap E(\mathcal{C}_i)|\geq \tau_{d}(\mathcal{C}_i)=2d+1$ for $i=1, 2$. 
Therefore, $|T'|=|T'\cap E(\mathcal{C}_1)|+|T'\cap E(\mathcal{C}_2)|\geq 2(2d+1)= 4d+2$. 
Consequently, $\tau_{d}(C_m\square C_n)= 4d+2$ for $\frac{m-1}{2}\leq d\leq \frac{mn-1}{2}$.
\end{proof}

\begin{Lem}\rm{\cite{HTLM2017}}
Let $m\geq 5$ and $n\geq 3$ be two odd integers. 
Then $mp(C_m\square C_n)=7$.
\end{Lem}

\begin{Lem}
For odd integers $m\geq 5$ and $n\geq 3$, 
$\tau_{1}(C_m\square C_n)=7$.
\end{Lem}
\begin{proof}
As $\tau_{1}(C_m\square C_n)=mp(C_m\square C_n)$, 
$\tau_{1}(C_m\square C_n)=7$ by Lemma 4.6.
\end{proof}

\begin{Lem}
For odd integers $m\geq 7$ and $n\geq 3$, 
$4d+2 \leq \tau_{d}(C_m\square C_n)\leq 5d+2$ if $2\leq d\leq \frac{m-3}{2}$.
\end{Lem}
\begin{proof}
By Lemma 4.5, we have $\tau_{d}(C_m\square C_n)\geq 4d+2$. 
Next, since $d\leq \frac{m-3}{2}$, 
we construct a set $T$ of $C_m\square C_n$ by taking 
$\partial(x_{11})\cup \partial(x_{12})\cup \partial(x_{21})\cup \partial(x_{22}) 
\cup \dots \cup \partial(x_{d1})\cup \partial(x_{d2})$. 
Clearly, $\nu(C_m\square C_n-T)= \frac{mn-1}{2}-d$. 
Thus, for every almost perfect matching $M$ of $C_m\square C_n$, we have $|T\cap M|\geq d$. 
Therefore, $T$ is a $d$-transversal set with $|T|=5d+2$. 
Combining these results, we conclude $4d+2 \leq \tau_{d}(C_m\square C_n)\leq 5d+2$ for $2\leq d\leq \frac{m-3}{2}$.
\end{proof}

The upper bound in Lemma 4.8 is not necessarily achieved for $d=2$. 
For example, let $T$ be the set of thick edges as illustrated in Fig.\hspace{0.06cm}\ref{J}. 
We have $\nu(C_7\square C_3-T)=8$. Consequently, 
for every almost perfect matching $M$ of $C_7\square C_3$, $|T\cap M|\geq 2$. 
Therefore, $T$ is a $2$-transversal set of size 11.

\begin{figure}[!htbp]
\begin{center}
\includegraphics[totalheight=3cm]{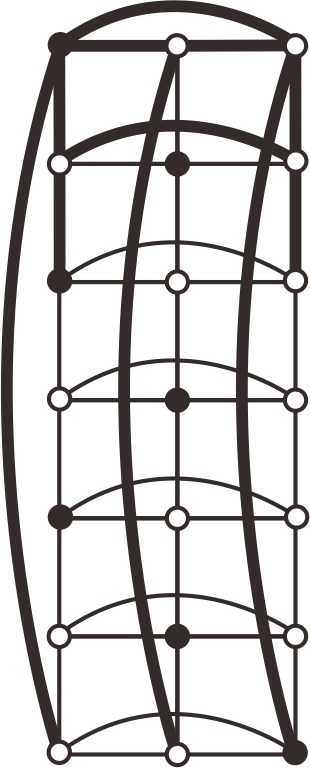}
\caption{\label{J}\small{A $2$-transversal set (thick edges) of size 11.}}
\end{center}
\end{figure}

From Lemmas 4.1, 4.4, 4.5, 4.7 and 4.8, 
we have the following result.

\begin{The}
Let $C_m\square C_n$ be a toroidal grid graph with $m\geq 3$ and $n\geq 3$. Then\\
1. for even $m\geq 4$ and even $n\geq 4$, $\tau_{d}(C_m\square C_n)=4d$ if $1\leq d\leq \frac{mn}{2}$,\\
2. for even $m\geq 4$ and odd $n\geq 3$, $\tau_{d}(C_m\square C_n)=4d$ if $1\leq d\leq \frac{mn}{2}$,\\
3. for odd $m\geq 3$ and odd $n\geq 3$,\\
(a) \hspace{0.1cm}if $m\geq 5$ and $n\geq 3$, then $\tau_{1}(C_m\square C_n)=7$,\\
(b) \hspace{0.1cm}if $m\geq 7$ and $n\geq 3$, then $4d+2 \leq \tau_{d}(C_m\square C_n)\leq 5d+2$ for $2\leq d\leq \frac{m-3}{2}$,\\
(c) \hspace{0.1cm}if $m\geq 3$ and $n\geq 3$, then $\tau_{d}(C_m\square C_n)= 4d+2$ for $\frac{m-1}{2}\leq d\leq \frac{mn-1}{2}$.\\
\end{The}

\section*{Data availability}
No data was used for the research described in the article.

\section*{Declarations}
\noindent\textbf{Conflict of interest} 
The authors declare that they have no Conflict of interest.

%
%


\end{document}